\documentclass[a4paper,11pt]{amsart}
\usepackage{mathrsfs}
\usepackage{appendix}
\usepackage{amssymb}
\usepackage{fancyhdr}
\usepackage{charter}
\usepackage{typearea}
\usepackage{color}
\usepackage{amsmath,amstext,amsthm,amscd}
\usepackage{hyperref}

\usepackage[a4paper,top=2.5cm,bottom=2.5cm,left=2cm,right=2cm]{geometry}

\makeatletter
      \def\@setcopyright{}
      \def\serieslogo@{}
      \makeatother
\newcommand{\boldsym}[1]{\boldsymbol{#1}}
\newcommand\bn{\boldsym{n}}

\newcommand{\N}{\mathbb N}
\newcommand{\dbar}{\partial}
\newcommand{\ddbar}{\overline\partial}

\newcommand{\ol}{\overline}

\DeclareMathOperator{\Ker}{Ker}

\newcommand{\cali}[1]{\mathscr{#1}}
\newcommand{\cH}{\cali{H}}
\newcommand{\cE}{\cali{E}}

\newcommand{\mT}{\mathcal{T}}
\newcommand{\R}{\mathbb{R}}

\newcommand{\ov}{\overline}
\newcommand{\til}[1]{\widetilde{#1}}
\newcommand{\wi}{\widetilde}
\def\Im{{\rm Im}}

\DeclareMathOperator{\supp}{supp}
\DeclareMathOperator{\Dom}{Dom}
\DeclareMathOperator{\rank}{rank}
 \def\cC{\mathscr{C}}
\def\cL{\mathscr{L}}
\theoremstyle{plain}
\newtheorem{theorem}{Theorem}[section]
\newtheorem{lemma}[theorem]{Lemma}
\newtheorem{corollary}[theorem]{Corollary}
\newtheorem{proposition}[theorem]{Proposition}

\numberwithin{equation}{section}

\begin{document}

\title[On Bergman kernel functions and weak Morse inequalities]
{On Bergman kernel functions and weak holomorphic Morse inequalities
 \let\thefootnote\relax\footnotetext{Xiaoshan Li was supported by NSFC Grant No.11871380. The second author was supported by NSFC Grant No. 12001549 and Guangdong Basic and Applied Basic Research Foundation Grant No. 2019A1515110250.}}

\author[]{Xiaoshan Li}
\address{School of Mathematics
and Statistics, Wuhan University, Wuhan 430072, Hubei, China}
\email{xiaoshanli@whu.edu.cn}

\author[]{Guokuan Shao}
\address{School of Mathematics (Zhuhai), Sun Yat-sen University, Zhuhai 519082, Guangdong, China}
\email{shaogk@mail.sysu.edu.cn}

\author[]{Huan Wang}
\address{Tong Jing Nan Lu 798, Suzhou, P. R. China.}
\email{huanwang2016@hotmail.com}
\date{}

\begin{abstract}
We give simple and unified proofs of weak holomorhpic Morse inequalities on complete manifolds, $q$-convex manifolds, pseudoconvex domains, weakly $1$-complete manifolds and covering manifolds. This paper is essentially based on the asymptotic Bergman kernel functions and the Bochner-Kodaira-Nakano formulas.
\end{abstract}
\maketitle
%\tableofcontents
\section{Introduction}

The history of holomorphic Morse inequalities initiated from the seminal work \cite{Dem:85} influenced by Siu's solution of Grauert-Riemenschneider conjecture and the classical Morse inequalities. It provides a flexible way to produce holomorphic sections of high tensor powers of line bundle under rather mild positivity assumption. Adjacent to the contribution of Shiffman-Ji, Takayama and Bonavero, one of the fundamental application is the full characterization of Moishezon manifolds and big line bundles.  Since then the holomorphic Morse inequalities attracted intensively study during the past two decades, see \cite{MM07} and reference therein for a comprehensive exposition.  Recently the Morse inequalities have been proved in new context, such as \cite{HM12,HsiaoLi:16,HHLS20} for CR manifolds and \cite{Pul:17,Pul:18} for the $G$-invariant case in the spirit of geometric quantization and the orbifold case.

Holomorphic Morse inequalities are global results encoded with local features, which can be obtained by the study of the behaviors of heat kernels, Bergman kernels or Szegő kernels. In this paper, we consider the asymptotic Bergman kernel functions \cite{Be04,HM:14} inspired by Berndtsson's work \cite{BB:02}. Comparing to the previous proof based on the spectral theory of Kodaira Laplacian, this method involves relatively elementary techniques.
 The key observation of this paper is the refinement of the $L^2$ weak Morse inequality as follows.
\begin{proposition}\label{prop_main}
		Let $(X,\omega)$ be a Hermitian manifold of dimension $n$ and let $(L,h^L)$ and $(E,h^E)$ be holomorphic Hermitian line bundles on $X$. Let $0\leq q\leq n$.
		Suppose there exist a compact subset $K\subset X$ and $C_0>0$ such that,
		for sufficiently large $k$, we have
		\begin{equation}\label{eq_ofe}
		\left(1-\frac{C_0}{k}\right)||s||^2\leq \frac{C_0}{k}\left(||\ddbar^E_ks||^2+||\ddbar^{E*}_{k}s||^2\right)+\int_{K} |s|^2 dv_X
		\end{equation}
		for $s\in \Dom(\ddbar^E_k)\cap \Dom(\ddbar^{E*}_{k})\cap L^2_{0,q}(M,L^k\otimes E)$.
	Then  we have
	\begin{equation}
	\limsup_{k\rightarrow \infty}n!k^{-n}\dim H_{(2)}^{q}({X},{L}^k\otimes E)\leq \int_{K(q)}(-1)^q c_1(L,h^L)^n.
	\end{equation}
\end{proposition}

 Thanks to the optimal fundamental estimate \eqref{eq_ofe} and the asymptotic estimate of the upper bound of the Bergman kernel function, we can directly integrate the Bergman kernel function over a compact subset to get the dimension of harmonic space. As applications, we improved weak holomorphic Morse inequalities in a uniform way as follows. It is remarkable that our new proof is direct, and moreover, the upper-bound of asymptotic dimension of cohomology are sharper than the previous in literatures (Explicitly we replace relatively compact domain $U$ including $K$ inside by $K$). Based on the Bochner-Kodaira-Nakano formulas, we have the privilege to give uniform and simple proofs of weak holomorhpic Morse inequalities in various situations.

\begin{theorem}[Ma-Marinescu]\label{thm_complete}
	Let $(X,\Theta)$ be a complete Hermitian manifold of dimension $n$. Let $(L,h^L)$ be a holomorphic Hermitian line bundle  on $X$ such that $\Theta=c_1(L,h^L)$ on $X\setminus M$ for a compact subset $M$. Then for each $1\leq q\leq n$, we have
	\begin{equation}
	\limsup_{k\rightarrow \infty}n!k^{-n}\dim H^{q}_{(2)}({X},{L}^k\otimes K_X)\leq \int_{M(q)}(-1)^q c_1(L,h^L)^n.
	\end{equation}
\end{theorem}

  \begin{theorem}[Bouche]\label{thm_q_convex}
  	Let $X$ be a $q$-convex manifold of dimension $n$ and $1\leq q \leq n$. Let $(L,h^L)$ and $(E,h^E)$ be holomorphic Hermitian line bundles on $X$. Suppose $R^L$ has at least $n-s+1$ non-negative eigenvalues on $X\setminus M$ for a compact subset $M$ with $1\leq s\leq n$.
  	Then for each $s+q-1\leq j\leq n$, we have
  	\begin{equation}
  	\limsup_{k\rightarrow \infty}n!k^{-n}\dim H^{j}({X},{L}^k\otimes E)\leq \int_{M(j)}(-1)^j c_1(L,h^L)^n.
  	\end{equation}
  \end{theorem}

\begin{theorem}[Ma-Marinescu]\label{thm_psc}
	Let $M\Subset X$ be a smooth pseudoconvex domain in a complex manifold $X$ of dimension $n$. Let $(L,h^L)$ and $(E,h^E)$ be holomorphic Hermitian line bundles on $X$.  Let $(L,h^L)$ be positive in a neighbourhood of the boundary $bM$ of $M$.
	Then for any $q\geq 1$, we have
	\begin{equation}
	\limsup_{k\rightarrow \infty}n!k^{-n}\dim H^{q}_{(2)}({M},{L}^k\otimes E)\leq \int_{M(q)}(-1)^q c_1(L,h^L)^n.
	\end{equation}
\end{theorem}

\begin{theorem}[Marinescu] \label{thm_w1c}
	Let $X$ be a weakly $1$-complete manifold of dimension $n$. Let $(L,h^L)$ and $(E,h^E)$ be a holomorphic Hermitian line bundles on $X$. Assume $K\subset X_c:=\{\varphi<c\}$ is a compact subset and $(L,h^L)$ is Griffith $q$-positive on $X\setminus K$ with $q\geq 1$. Then, there exits a Hermitian metric $\omega$ on $X$ and as $k\rightarrow \infty$, for any $j\geq q$,
	\begin{equation}
		\limsup_{k\rightarrow \infty}n!k^{-n}\dim H^j_{(2)}(X_c,L^k\otimes E)\leq \rank(E)\int_{K(j)}(-1)^j c_1(L,h^L)^n.
	\end{equation}
	In particular, if $L>0$ on $X\setminus K$, the inequalities hold for $H^j(X,L^k\otimes E)$ with all $j\geq 1$.
\end{theorem}

Without the optimal fundamental estimate \ref{eq_ofe},
we also can obtain:

\begin{theorem}[Chiose-Marinescu-Todor]\label{thm_cover}
	Let $(\widetilde X,\widetilde \omega)$ be a Hermitian manifold of dimension $n$ on which a discrete
	group $\Gamma$ acts holomorphically, freely and properly such that $\widetilde{\omega}$ is a $\Gamma$-invariant Hermitian
	metric and the quotient $X=\widetilde X/\Gamma$ is compact. Let $(\widetilde L,h^{\widetilde L})$ and $(\widetilde E,h^{\widetilde E})$ be $\Gamma$-invariant
	holomorphic Hermitian line bundles on $\widetilde X$. Then for $0\leq q\leq n$,
	\begin{equation}
	\limsup_{k\rightarrow \infty}n!k^{-n}\dim_{\Gamma}\overline H^{q}_{(2)}(\widetilde{X},\widetilde{L}^k\otimes \til{E})\leq \int_{X(q)}(-1)^q c_1(L,h^L)^n.
	\end{equation}
\end{theorem}
 This paper is organized as follows. In Sect. \ref{Sec_pre} we introduce necessary facts related to $L^2$-cohomology and asymptotic Bergman kernel functions. In Sec. \ref{Sec_l2wmi} we prove the refinement of $L^2$ weak Morse inequalities Proposition \ref{prop_main} and Theorem \ref{thm_complete}--\ref{thm_cover} as applications. In Sec. \ref{Sec_specfct} further results linked to asymptotics of spectral function of lower energy forms for Kodaira Laplacian are given. The general framework and techniques are due to \cite{MM07,HM:14}.

\section{Preliminaries and notations} \label{Sec_pre}

Let $(X, \omega)$ be a Hermitian manifold of dimension $n$ and $(F, h^F)$ and $(L,h^L)$ be holomorphic Hermitian vector bundle on $X$ with $\rank(L)=1$.
Let $\Omega^{p,q}(X, F)$ be the space of smooth $(p,q)$-forms on $X$ with values in $F$ for $p,q\in \N$. If $\rank(F)=1$, the curvature of $(F, h^F)$ is defined by $R^F=\ddbar\dbar \log|s|^2_{h^{F}}$ for any local holomorphic frame $s$ and the Chern-Weil form of the first Chern class of $F$ is denoted by $c_1(F, h^F)=\frac{\sqrt{-1}}{2\pi}R^F$, which is a real $(1,1)$-form on $X$. If $\rank(F)=1$, by identifying $R^F(x)$, $x\in X$, with a Hermitian matrix via the Hermitian metric $\omega$, we can consider the numbers of positive, negative and zero eigenvalues of $R^F(x)$ , which are independent of the choice of $\omega$. The volume form is given by $dv_{X}:=\omega_n:=\frac{\omega^n}{n!}$.

\subsection{$L^2$-coholomogy}
Let $\Omega^{p,q}_0(X, F)$ be the subspace of
$\Omega^{p,q}(X, F)$ consisting of elements with compact support.
The $L^2$-scalar product on $\Omega^{p,q}_0(X, F)$ is given by
\begin{equation}\label{e:sp}
\langle s_1,s_2 \rangle=\int_X \langle s_1(x), s_2(x) \rangle_h dv_X(x)
\end{equation}
where $\langle\cdot,\cdot\rangle_h:=
\langle\cdot,\cdot\rangle_{h^F,\omega}$
is the pointwise Hermitian inner product induced by $\omega$ and $h^F$.
We denote by $L^2_{p,q}(X, F)$, the $L^2$ completion of $\Omega^{p,q}_0(X, F)$.

Let $\ddbar^{F}: \Omega_0^{p,q} (X, F)\rightarrow L^2_{p,q+1}(X,F) $ be the Dolbeault operator and let  $ \ddbar^{F}_{\max} $ be its maximal extension. From now on we still denote the maximal extension by $ \ddbar^{F} :=\ddbar^{F}_{\max} $ and  the corresponding Hilbert space adjoint by $\ddbar^{F*}:=\ddbar^{F*}_H:=(\ddbar^{F}_{\max})_H^*$. We write $\ddbar^F_k:=\ddbar^{L^k\otimes F}$ for simplification.
Consider the complex of closed, densely defined operators
$L^2_{p,q-1}(X,F)\xrightarrow{\ddbar^{F}}L^2_{p,q}(X,F)\xrightarrow{\ddbar^{F}} L^2_{p,q+1}(X,F)$,
then $(\ddbar^{F})^2=0$. By \cite[Proposition 3.1.2]{MM07}, the operator defined by
\begin{eqnarray}\label{eq40}\nonumber
\Dom(\square^{F})&=&\{s\in \Dom(\ddbar^{F})\cap \Dom(\ddbar^{F*}):
\ddbar^{F}s\in \Dom(\ddbar^{F*}),~\ddbar^{F*}s\in \Dom(\ddbar^{F}) \}, \\
\square^{F}s&=&\ddbar^{F} \ddbar^{F*}s+\ddbar^{F*} \ddbar^{F}s \quad \mbox{for}~s\in \Dom(\square^{F}),
\end{eqnarray}
is a positive, self-adjoint extension of Kodaira Laplacian, called the Gaffney extension.
%{definition}
	The space of harmonic forms $\cH^{p,q}(X,F)$ is defined by
	\begin{equation}\label{eq45}
	\cH^{p,q}(X,F):=\Ker(\square^{F})\cap  L^2_{p,q}(X,F) =\{s\in \Dom(\square^{F})\cap L^2_{p,q}(X, F): \square^{F}s=0 \}.
	\end{equation}
	The $q$-th reduced (resp. non-reduced) $L^2$-Dolbeault cohomology are defined by, respectively,
	\begin{equation}\label{eq46}
	\overline{H}^{0,q}_{(2)}(X,F):=\dfrac{\Ker(\ddbar^{F})\cap  L^2_{0,q}(X,F) }{[ \Im( \ddbar^{F}) \cap L^2_{0,q}(X,F)]}, \quad H^{0,q}_{(2)}(X,F):=\dfrac{\Ker(\ddbar^{F})\cap  L^2_{0,q}(X,F) }{ \Im( \ddbar^{F}) \cap L^2_{0,q}(X,F)},
	\end{equation}
	where $[V]$ denotes the closure of the space $V$.
According to the general regularity theorem of elliptic operators,
$s\in \cH^{p,q}(X,F) $ implies $s\in\Omega^{p,q}(X,F)$. By weak Hodge decomposition (cf.\ \cite[(3.1.21) (3.1.22)]{MM07}),
we have a canonical isomorphism for any $q\in \N$,
\begin{equation}\label{eq47}
\overline{H}^{0,q}_{(2)}(X,F)\cong \cH^{0,q}(X,F),
\end{equation} which associates to each cohomology class its unique harmonic representative. The sheaf cohomology of holomorphic sections of $F$ is isomorphic to the Dolbeault cohomology, $H^\bullet(X,F)\cong H^{0,\bullet}(X,F)$.

\subsection{Optimal fundamental estimates}
We say the \textbf{fundamental estimate} holds in bidegree $(0,q)$ for forms with values in $F$ with $0\leq q\leq n$, if there exist a compact subset $K\subset X$ and $C>0$ such that, for $s\in \Dom(\ddbar^F)\cap\Dom(\ddbar^{F,*})\cap L^2_{0,q}(X,F)$, we have
\begin{equation}
\|s\|^2\leq C\left(\|\ddbar^F s\|^2+\|\ddbar^{F*}\|^2+\int_K|s|^2dv_X\right).
\end{equation}
In this case, we have $\ov H_{(2)}^{0,q}(X,F)\cong H^{0,q}_{(2)}(X,F)$, see \cite[Theorem 3.1.8]{MM07}. We say the \textbf{optimal fundamental estimate} holds in bidegree $(0,q)$ for forms with values in $L^k\otimes F$ with $0\leq q\leq n$, if there exist a compact subset $K\subset X$ and $C_0>0$ such that, for sufficiently large $k$ we have
for $s\in \Dom(\ddbar^F_k)\cap \Dom(\ddbar^{F*}_{k})\cap L^2_{0,q}(M,L^k\otimes F)$,
\begin{equation}
\left(1-\frac{C_0}{k}\right)||s||^2\leq \frac{C_0}{k}\left(||\ddbar^F_ks||^2+||\ddbar^{F*}_{k,H}s||^2\right)+\int_{K} |s|^2 dv_X.
\end{equation}
With respect to forms with values in $L^k\otimes F$, optimal fundamental estimate holds implies fundamental estimate holds for sufficiently large $k$. The optimal means the coefficient of the term $\int_K |s|^2 dv_X$ is $1$, which is crucial in our approach.
In the sequel we will see, besides of the trivial case $X=K$, the optimal fundamental estimate holds on various possibly non-compact complex manifolds.

For the purpose of the optimal fundamental estimates fulfilled in various situations, we need the following Bochner-Kodaira-Nakano formulas in the general form \cite{Dem:86}, which is extremely useful in complex geometry of vector bundles \cite{Dem:12} and has geometric applications on CR manifolds \cite{HMW20}.
See \cite[Theorem 1.4.12]{MM07} for a detailed proof.

\begin{theorem}[Demailly's Bochner-Kodaira-Nakano formula]
	On $\Omega^{\bullet,\bullet}(X,F)$, we have
	\begin{equation}\label{eq_bkn}
	\square^F=\overline{\square}^F+[\sqrt{-1}R^F,\Lambda]+[(\nabla^F)^{1,0},\mT^*]+[(\nabla^F)^{0,1},\overline{\mT}^*].
	\end{equation}
\end{theorem}
Also the Bochner-Kodaira-Nakano formula with boundary term due to Andreotti-Vesentini-Griffiths is vital to our approach. The following inequality is a geometric version of the Morrey-Kohn-H\"{o}rmander estimate, which is important in the solution of the $\ddbar$-Neumann problem.

\begin{theorem}[{\cite[Corollary 1.4.22]{MM07}}]\label{BKNwithbd}
	Let $M$ be a smooth relatively compact domain in a Hermitian manifold $(X,\omega)$. Let $\rho\in \cC^\infty(X)$ such that $M=\{x\in X: \rho(x)< 0\}$ and $|d\rho|=1$ on the boundary $bM$. Let $(F,h^F)$ be a holomorphic Hermitian vector bundles on $X$. Then
	for any $s\in B^{0,p}(M,F)$, $0\leq p\leq n$,
	\begin{equation}
	\begin{split}
	\frac{3}{2}\left(||\ddbar^F s||^2+||\ddbar^{F*}s||^2\right)
	&\geq \frac{1}{2}||(\nabla^{\til{F}})^{1,0*}\til{s}||^2+\left\langle  R^{F\otimes K^*_X}(w_j,\ov{w}_k)\ov{w}^k\wedge i_{\ov{w}_j} s,s\right\rangle\\
	&+\int_{bM}\cL_{\rho}(s,s)dv_{bM}-\frac{1}{2}\left(||\mT^*\til{s}||^2+||\ov{\mT}\til{s}||^2+||\ov{\mT}^* \til{s}||^2\right),
	\end{split}
	\end{equation}
	where $\{ w_j \}_{j=1}^n$ is a local orthonormal frame of $T^{(1,0)}X$ with dual frame $\{ w^j\}_{j=1}^n$ of $T^{(1,0)*}X$,  $\cL_{\rho}(\cdot,\cdot):=(\dbar\ddbar \rho)(w_k,\ov w_j)\langle \ov w^j\wedge i_{\ov w_k} \cdot,\cdot\rangle_h$ is Levi form of $bM$, $\mT:=[i_\omega,\dbar\omega]$ is the Hermitian torsion operator and $\mT^*$ is its formal adjoint, $\wi{F}:=F\otimes K^*_X$, $\nabla^{\wi F}$ is Chern connection and $\wi s:=(w^1\wedge\cdots\wedge w^n\wedge s)\otimes (w_1\wedge\cdots\wedge w_n)$.
\end{theorem}

\subsection{Asymptotic Bergman kernel function of line bundles}
Berman proved a local version of weak holomorphic Morse inequalities, which holds regardless of compactness or completeness. Refer to \cite[Theorem 1.1, Remark 1.3]{Be04} and \cite[Corollary 1.4]{HM:14} for details.
Let $(X,\omega)$ be a Hermitian manifold of dimension $n$. Let $(L,h^L)$ and $(E,h^E)$ be holomorphic Hermitian line bundles on $X$.
Let $\{{s}^k_j\}_{j\geq 1}$ be an orthonormal basis of $\cH^{0,q}({X},{L}^k\otimes E)$, $0\leq q\leq n$, and $|\cdot|:=|\cdot|_{{h}_k,{\omega}}$ the point-wise Hermitian norm. The Bergman kernel function on $X$ is defined by
\begin{equation}
{B}^{q}_k({x})=\sum_{j}|s^k_j({x})|,
\;{x}\in {X}.
\end{equation}

\begin{theorem}[{\cite[Theorem 1.1]{Be04}}{\cite[Corollary 1.4]{HM:14}}]\label{thm_lwhmi_b}
	For any $x\in X$, we have
	\begin{equation}\label{eq_berman_mz}
	\limsup_{k\rightarrow \infty} k^{-n}B^q_k(x)\leq (-1)^q 1_{{X}(q)}\frac{c_1({L},h^{{L}})^n}{{\omega}^n}({x}).
	\end{equation}
\end{theorem}

\begin{theorem}[{\cite[Theorem 4.3]{HM:14} \cite[Corollary 3.3]{Be04}}]\label{thm_lub_b} If $K\subset X$ be a compact subset, then there exist $C>0$ and  $k_0\in\N$, such that for any $x\in K$ and $k>k_0$,
	\begin{equation}
	k^{-n}B^q_k(x)\leq  C.
	\end{equation}
\end{theorem}

	The local estimate of Bergman function (\ref{eq_berman_mz}) was inspired by \cite{BB:02}, in which the refined estimate $k^{-n+q}B^q_k(x)\leq C$ holds for semipositive line bundles on compact complex manifolds. Furthermore, it still holds on arbitrary manifolds, when restricting Bergman kernel function on compact subsets, see \cite{Wh:16,Wh:21a,Wh:21b} for parallel results of Proposition \ref{prop_main} and Theorem \ref{thm_complete}--\ref{thm_cover}.

\section{$L^2$ weak Morse inequalities}\label{Sec_l2wmi}
The point of passing from the compact manifold to the non-compact is, under appropriate assumption on the positivity of line bundle $L$ or the convexity of manifold $X$, the norm of harmonic forms with values in $L^k\otimes E$ decay to zero as $k\rightarrow\infty$ outside of a compact subset. As a consequence, the computation of the dimension of cohomology concentrates on a compact subset.

\begin{proposition}\label{cor_l2_morse}
	Let $(X,\omega)$ be a Hermitian manifold of dimension $n$ and let $(L,h^L)$ and $(E,h^E)$ be holomorphic Hermitian line bundles on $X$.
	Suppose there exist a sequence of real numbers $a_k\in\R$ with $k\in\N$, and a compact subset $K\subset X$ such that:
		(1) $\limsup_{k\rightarrow \infty}a_k=1$;
		(2) For sufficiently large $k$, $\| s\|^2\leq a_k \int_K |s|^2 dv_X$ for each $s\in\Ker(\ddbar_k^{ E})\cap\Ker(\ddbar^{E*}_{k})\cap L^2_{0,q}(X,L^k\otimes E)$ with $0\leq q\leq n$.
	Then,  we have
	\begin{equation}
	\limsup_{k\rightarrow \infty}n!k^{-n}\dim \cH^{0,q}({X},{L}^k\otimes E)\leq \int_{K(q)}(-1)^q c_1(L,h^L)^n.
	\end{equation}
	In particular, Proposition \ref{prop_main} holds true.
\end{proposition}
\begin{proof}
	  By Theorem \ref{thm_lub_b} and applying Fatou's lemma to the sequence of non-negative functions $C-k^{-n}B^q_k$ on $K$, as well as the finiteness of the volume $\int_K C dv_X<\infty$, we have
	\begin{equation} \label{eq_fatou}
	\limsup_{k\rightarrow \infty} \int_K k^{-n} {B}^{q}_k({x}) dv_{{X}}({x})
	\leq \int_K \limsup_{k\rightarrow \infty}k^{-n}{B}^{q}_k({x})dv_{{X}}({x}).
	\end{equation}
	From the assumptions (1), (2) and (\ref{eq_berman_mz}) (\ref{eq_fatou}), it follows that
	\begin{equation}
	\begin{split}
	&\limsup_{k\rightarrow \infty}\left(k^{-n}\dim\cH^{0,q}(X,L^k\otimes E)\right)\leq\limsup_{k\rightarrow \infty}\left(k^{-n}a_k \int_{K}B_k^q(x)dv_X(x)\right)\\
	&\leq\left(\limsup_{k\rightarrow \infty} a_k\right)\left(\limsup_{k\rightarrow \infty} \int_{K}k^{-n}B_k^q(x)dv_X(x)\right)\leq\int_{K}\limsup_{k\rightarrow \infty}k^{-n} B_k^q(x)dv_X(x)\\
	&\leq \int_{K(q)}(-1)^q \frac{c_1({L},h^{{L}})^n}{n!}.
	\end{split}
	\end{equation}
In particular, Proposition \ref{prop_main} holds true from \eqref{eq_ofe} and $H_{(2)}^{q}({X},{L}^k\otimes E)\cong \cH^{0,q}({X},{L}^k\otimes E)$.
\end{proof}

\subsection{$q$-convex manifolds}

	A complex manifold $X$ of dimension $n$ is called $q$-convex if there exists a smooth function $\varrho\in \cC^\infty(X,\R)$ such that the sublevel set $X_c=\{ \varrho<c\}\Subset X$ for all $c\in \R$ and the complex Hessian $\dbar\ddbar\varrho$ has $n-q+1$ positive eigenvalues outside a compact subset $K\subset X$. Here $X_c\Subset X$ means that the closure $\overline{X}_c$ is compact in $X$. We call $\varrho$ an exhaustion function and $K$ exceptional set.

Let $M$ be a relatively compact domain in $X$. Let $\rho$ be a defining function of $M$ satisfying $M=\{ x\in X: \rho(x)<0 \}$ and $|d\rho|=1$ on $bM$, where the pointwise norm $|\cdot|$ is given by $g^{TX}$ associated to $\omega$. Let $e_{\bn} \in TX$ be the inward pointing unit normal at $bM$ and $e_{\bn}^{(0,1)}$ its projection on $T^{(0,1)}X$. In a local orthonormal frame $\{ w_1,\cdots,\omega_n \}$ of $T^{(1,0)}X$, we have $e_{\bn}^{(0,1)}=-\sum_{j=1}^n w_j(\rho)\ov w_j$. Let $B^{0,q}(X,F):=\{ s\in \Omega^{0,q}(\ov M, F): i_{e_{\bn}^{(0,1)}}  s=0 ~\mbox{on}~bM \}$. We have $B^{0,q}(M,F)=\Dom(\ddbar_H^{F*})\cap \Omega^{0,q}(\overline{M},F)$ and the Hilbert space adjoint $\ddbar_H^{F*}$ of $\ddbar^F$ coincides with the formal adjoint $\ddbar^{F*}$ of $\ddbar^F$ on $B^{0,q}(M,F)$, see \cite[Proposition 1.4.19]{MM07}. We consider the operator $\square_N s=\ddbar^{F}\ddbar^{F*}s+\ddbar^{F*}\ddbar^{F}s$ for $s\in \Dom(\square_N):=\{s\in B^{0,q}(M,F): \ddbar^Fs\in B^{0,q+1}(M,F) \}$. The Friedrichs extension of $\square_N$ is a self-adjoint operator and is called the Kodaira Laplacian with $\ddbar$-Neumann boundary conditions, which coincides with the Gaffney extension of the Kodaira Laplacian, see \cite[Proposition 3.5.2]{MM07}. Note $\Omega^{0,\bullet}(\overline{M},F)$ is dense in $\Dom(\ddbar^F)$ in the graph-norms of $\ddbar^F$, and $B^{0,\bullet}(M,F)$ is dense in $\Dom(\ddbar^{F*}_H)$ and in $\Dom(\ddbar^F)\cap \Dom(\ddbar^{F*}_H)$ in the graph-norms of $\ddbar^{F*}_H$ and $\ddbar^F+\ddbar^{F*}_H$, respectively, see \cite[Lemma 3.5.1]{MM07}. Here the graph-norm is defined by $\|s\|+\|Rs\|$ for $s\in \Dom(R)$.

From now on let $X$ be a $q$-convex manifold of dimension $n$. Let $u_0<u<c<v$ such that the exceptional subset $K\subset X_{u_0}:=\{x\in X: \varrho(x)<{u_0} \}$.

Firstly, we choose now a metric on $X$ due to \cite[Lemma 3.5.3]{MM07} as follows.

\begin{lemma} \label{lowbd_rho_lem}
	For any $C_1>0$ there exists a metric $g^{TX}$ (with Hermitian form $\omega$) on $X$ such that for any $j\geq q$ and any holomorphic Hermitian vector bundle $(F,h^F)$ on $X$,
	\begin{equation}
	\langle (\dbar\ddbar\varrho)(w_l,\ol{w}_k)\ol{w}^k\wedge i_{\ol{w}_l}s,s \rangle_h\geq C_1|s|^2, \quad s\in \Omega^{0,j}_0(X_v\setminus \overline{X}_u,F),
	\end{equation}
	where $\{ w_l \}_{l=1}^n$ is a local orthonormal frame of $T^{(1,0)}X$ with dual frame $\{ w^l\}_{l=1}^n$ of $T^{(1,0)*}X$.
\end{lemma}

From now on we consider the $q$-convex manifold $X$ associated with the metric $\omega$ obtained above as a Hermitian manifold $(X,\omega)$. Note for arbitrary holomorphic vector bundle $F$ on a relatively compact domain $M$ in $X$, we simply use the notion $\ddbar^{F*}=\ddbar_H^{F*}$ on $B^{0,j}(M,F)$, $1\leq j\leq n$.

Secondly, we modify the prescribed hermitian metric $h^L$ on $L$.
Let $\chi(t)\in\cC^\infty(\R)$ such that $\chi'(t)\geq 0$, $\chi''(t)\geq 0$. We define a Hermitian metric $h^{L^k}_\chi:=h^{L^k}e^{-k\chi(\varrho)}$ on $L^k$ for each $k\geq 1$ and we set $L^k_\chi:=(L^k,h^{L^k}_\chi)$. Thus
\begin{equation}
R^{L^k_\chi}=kR^{L_\chi}=kR^L+k\chi'(\varrho)\dbar\ddbar\varrho+k\chi''(\varrho)\dbar\varrho\wedge\ddbar\varrho.
\end{equation}

\begin{lemma}\label{keylem}
	There exist $C_2>0$ and $C_3>0$ such that, if $\chi'(\varrho)\geq C_3$ on $X_v\setminus \overline{X}_u$, then
	\begin{equation}
	||s||^2\leq \frac{C_2}{k}( ||\ddbar^E_k s||^2+||\ddbar^{E*}_k s||^2 )
	\end{equation}
	for $s\in B^{0,j}(X_c,L^k\otimes E)$ with $\supp(s)\subset X_v\setminus \overline{X}_u$, $j\geq q$ and $k\geq 1$,
	where the $L^2$-norm $||\cdot||$ is given by $\omega$, $h^{L^k}_\chi$ and $h^E$ on $X_c$.
\end{lemma}
\begin{proof}
	We follow \cite[(3.5.19)]{MM07}.
	Assume $c$ is a regular value of $\varrho$ by Sard's theorem. Let $\varrho_1\in \cC^\infty(X)$ be a defining function of $X_c$ such that $\varrho_1=\frac{\varrho-c}{|d\varrho|}$ near $bX_c$.
	Thus $X_c=\{ x\in X: \varrho_1(x)<0 \}$ and $|d\varrho_1|=1$ on $bX_c$. 	
	Theorem \ref{BKNwithbd} implies that for $s\in B^{0,p}(X_c,L^k\otimes E)$, $1\leq p\leq n$, with respect to $\omega$, $h_{\chi}^L$ and $h^E$, we have
	\begin{equation}\label{eq_BKN}
	\begin{split}
	\frac{3}{2}(||\ddbar^E_k s||^2+||\ddbar^{E*}_k s||^2)
	\geq &\langle  R^{L_\chi^k\otimes E\otimes K^*_X}(w_j,\ov{w}_k)\ov{w}^k\wedge i_{\ov{w}_j} s,s\rangle
	+\int_{bX_c}\cL_{\varrho_1}(s,s)dv_{bX_c}\\
	&-\frac{1}{2}(||\mT^*s||^2+||\ov{\mT}\til{s}||^2+||\ov{\mT}^* \til{s}||^2).
	\end{split}
	\end{equation}
	Note $R^{L_\chi^k\otimes E\otimes K^*_X}=kR^L+k\chi'(\varrho)\dbar\ddbar\varrho+k\chi''(\varrho)\dbar\varrho\wedge\ddbar\varrho+R^{E\otimes K^*_X}$ and $\sqrt{-1}\chi''(\varrho)\dbar\varrho\wedge\ddbar\varrho\geq 0$.	
	From the facts that $\ov{X}_v\subset X$ is compact and Lemma \ref{lowbd_rho_lem}, for $q\leq j\leq n$ there exist $C_L\geq 0$, $C_4\geq 0$ and $C_5\geq 0$ such that, for any $s\in B^{0,j}(X_c,L^k\otimes E)$ with supp$(s)\subset X_v\setminus \overline{X}_u$ and $k\geq 1$, we have
	\begin{equation}\label{eq_posline}
	\begin{split}
	&\langle R^L(w_j,\ov{w}_k)\ov{w}^k\wedge i_{\ov{w}_j} s,s\rangle\geq- C_L||s||^2,
	\\
	&\langle \dbar\ddbar\varrho(w_j,\ov{w}_k)\ov{w}^k\wedge i_{\ov{w}_j} s,s\rangle_h\geq C_1|s|^2,
    \\
	&\langle  R^{E\otimes K_X^*}(w_j,\ov{w}_k)\ov{w}^k\wedge i_{\ov{w}_j} s,s\rangle\geq -C_5||s||^2,
	\\
	&-\frac{1}{2}(||\mT^*\til{s}||^2+||\ov{\mT}\til{s}||^2+||\ov{\mT}^* \til{s}||^2)\geq -C_4||s||^2.
	\end{split}
	\end{equation} 	
	From $\varrho_1=\frac{\varrho-c}{|d\varrho|}$ near $bX_c$ and Lemma \ref{lowbd_rho_lem},
	\begin{equation}\label{eq_lrho1}
	\int_{bX_c}\cL_{\varrho_1}(s,s)dv_{bX_c}
	=\int_{bX_c}\frac{1}{|d\varrho|}\langle\dbar\ddbar\varrho(w_j,\ov{w}_k)\ov{w}^k\wedge i_{\ov{w}_j} s,s\rangle_h dv_{bX_c}
	\geq \int_{bX_c}\frac{C_1 |s|^2}{|d\varrho|} dv_{bX_c} \geq 0.
	\end{equation}
	for any $s\in B^{0,j}(X_c,L^k\otimes E)$, $j\geq q$, with $\supp(s)\subset X_v\setminus \overline{X}_u$ and $k\geq 1$.
	Finally we substitute (\ref{eq_posline}) and (\ref{eq_lrho1}  to (\ref{eq_BKN}) and obtain that
	\begin{eqnarray}\label{eqfe2}
	\frac{3}{2}(||\dbar^E_k s||^2+||\dbar^{E*}_k s||^2)\geq \int_{X_c} k(C_1\chi'(\varrho)-C_L)-C_5-C_4)|s|^2 dv_{X_c}.
	\end{eqnarray}
	We can set $C_2:=\frac{3}{2}$ and $C_3:=\frac{C_L+C_4+C_5+1}{C_1}$. Moreover, we can choose $C_1$, $C_L$, $C_4$ and $C_5$ works for all $q\leq j\leq n$, thus $C_2$ and $C_3$ also.
\end{proof}

\begin{lemma}\label{K'lem}
	Let $\epsilon>0$ satisfy $X_{c+\epsilon}\setminus \ov{X}_{c-\epsilon}:=\{ c-\epsilon<\varrho<c+\epsilon \}\Subset X_v\setminus \ov{X}_u$. Let $\phi\in \cC^\infty_0(X_v,\R)$ with $\supp(\phi)\subset X_v\setminus \ov{X}_u$ such that $0\leq \phi \leq 1$ and $\phi=1$ on $X_{c+\epsilon}\setminus \ov{X}_{c-\epsilon}$. Let $K':=\ov{X}_{c-\epsilon}:={\{\varrho\leq c-\epsilon\}}$.
	Then, for any $s\in B^{0,p}(X_c,L^k\otimes E)$, $1\leq p\leq n$, we have
	\begin{equation}
	||\phi s||^2\geq ||s||^2- \int_{K'}|s|dv_X,
	\end{equation}
	where the Hermitian norm $|\cdot|$ and the $L^2$-norm $||\cdot||$ are given by $\omega$, $h^{L^k}_\chi$ and $h^E$ on $X_c$.
\end{lemma}
\begin{proof}
	For $s\in B^{0,p}(X_c,L^k\otimes E)=\Dom(\ddbar^{E*}_k)\cap\Omega^{0,p}(\ov{X}_c,L^k\otimes E)$, $\phi s\in \Omega^{0,p}(\ov{X}_c,L^k\otimes E)$ and $i_{e_{\bn}^{(0,1)}}(\phi s)=i_{e_{\bn}^{(0,1)}}( s)=0$ on $bX_c$ by $\phi s=s$ on the neighbourhood $X_{c+\epsilon}\setminus \ov{X}_{c-\epsilon}$ of $bX_c$. Thus $\phi s \in B^{0,p}(X_c,L^k\otimes E)$ with $\supp(\phi s)\subset X_v\setminus \ov{X}_u$. We compute that
	\begin{equation}
	\begin{split}
	&||\phi s||^2=\int_{X_c}|\phi s|^2dv_X=\int_{X_c\setminus \ov{X}_u}|\phi s|^2dv_X
	=\int_{\{c-\epsilon<\varrho<c\}}|\phi s|^2 dv_X+\int_{u<\varrho\leq c-\epsilon}|\phi s|^2 dv_X\\
	&=\int_{\{c-\epsilon<\varrho<c\}}|s|^2 dv_X+\int_{\{u<\varrho\leq c-\epsilon\}}|\phi s|^2 dv_X
	\geq\int_{X_c\setminus \ov{X}_{c-\epsilon}}|s|^2 dv_X
	=||s||^2-\int_{K'}|s|^2 dv_X.
	\end{split}
	\end{equation}
\end{proof}

\begin{lemma}
	Let $\phi$ be in Lemma \ref{K'lem}, and let $\xi:=1-\phi$ and $C_1:=\sup_{x\in X_c}|d\xi(x)|_{g^{T^*X}}^2>0$. Then, for any $s\in B^{0,p}(X_c,L^k\otimes E)$, $1\leq p\leq n$, and $k\geq 1$, we have
	\begin{equation}\label{cuteq_k}
	\frac{1}{k}(||\ddbar^E_k (\phi s)||^2+||\ddbar^{E*}_k(\phi s)||^2)
	\leq \frac{5}{k}(||\ddbar^E_k s||^2+||\ddbar^{E*}_k s||^2)+\frac{12C_1}{k}||s||^2,
	\end{equation}
	where the $L^2$-norm $||\cdot||$ is given by $\omega$, $h^{L^k}_\chi$ and $h^E$ on $X_c$.
\end{lemma}
\begin{proof}
	We follow {\cite[(3.2.8)]{MM07}}.
	Since $\xi=1-\phi\in \cC^\infty(\ov{X}_c)$, we see $0\leq \xi\leq 1$, $\xi=1$ on $\ov X_u$ and $\xi=0$ on $\ov X_c\setminus \ov{X}_{c-\epsilon}$, thus $\xi\in \cC^\infty_0(X_c)$. Because $s\in B^{0,p}(X_c,L^k\otimes E)=\Omega^{0,p}(\ov{X}_c,L^k\otimes E)\cap\Dom(\ddbar^{E*}_k)$, $\xi s\in \Omega^{0,p}_0(X_c,L^k\otimes E)\subset B^{0,p}(X_c,L^k\otimes E)$.
	For simplifying notations, we use $\ddbar$ and $\ddbar^*$ instead of $\ddbar^E_k$ and $\ddbar^{E*}_k$ respectively, thus
	$||\ddbar (\xi s)||^2=||\ddbar\xi\wedge s+\xi \ddbar s||^2\quad\mbox{and} \quad
	||\ddbar^*(\xi s)||^2=||\xi\ddbar^*s-i_{\ddbar\xi}s||^2$.
	We also have that $||\ddbar\xi\wedge s||^2\leq C_1||s||^2$ and $||i_{\ddbar\xi}s||^2\leq C_1||s||^2$. Since $||A+B||^2\leq 3||A||^2+\frac{3}{2}||B||^2$,
	it follows
	\begin{equation}
	\begin{split}
	&||\ddbar s-\ddbar (\phi s)||^2+||\ddbar^* s-\ddbar^*(\phi s)||^2=||\ddbar (\xi s)||^2+||\ddbar^*(\xi s)||^2\\
	&\leq \frac{3}{2}(||\xi\ddbar s||^2+||\xi\ddbar^* s||^2) +3(||\ddbar\xi\wedge s||^2+||i_{\ddbar\xi}s||^2)
	\leq  \frac{3}{2}(||\ddbar s||^2+||\ddbar^* s||^2)+6C_1||s||^2.
	\end{split}
    \end{equation}
	By $2||A-B||^2\geq ||A||^2-2||B||^2$, we have $\frac{1}{2}||\ddbar(\phi s)||^2-||\ddbar s||^2\leq ||\ddbar s-\ddbar(\phi s)||^2$ and $\frac{1}{2}||\ddbar^*(\phi s)||^2-||\ddbar^* s||^2\leq ||\ddbar^* s-\ddbar^*(\phi s)||^2$, and thus
	$\frac{1}{2}(||\ddbar(\phi s)||^2+||\ddbar^*(\phi s)||^2)\leq \frac{5}{2}(||\ddbar s||^2+||\ddbar^* s||^2)+6C_1||s||^2$.
\end{proof}

\begin{proposition}\label{1coxFE}
	Let $X$ be a $q$-convex manifold of dimension $n$ with the exceptional set $K\subset X_c$. Then there exist a compact subset $K'\subset X_c$ with $K\subset K'$ and $C_0>0$ such that for sufficiently large $k$, we have
	\begin{equation}
	(1-\frac{C_0}{k})||s||^2\leq \frac{C_0}{k}(||\ddbar^E_ks||^2+||\ddbar^{E*}_{k,H}s||^2)+\int_{K'} |s|^2 dv_X
	\end{equation}
	for any $s\in \Dom(\ddbar^E_k)\cap \Dom(\ddbar^{E*}_{k,H})\cap L^2_{0,j}(X_c,L^k\otimes E)$ and $q\leq j \leq n$,
	where $\chi'(\varrho)\geq C_3$ on $X_v\setminus \overline{X}_u$ in Lemma \ref{keylem} and the $L^2$-norm is given by $\omega$, $h^{L^k}_\chi$ and $h^E$ on $X_c$.
\end{proposition}
\begin{proof}
	Since $B^{0,j}(X_c,L^k\otimes E)$ is dense in $\Dom(\ddbar^E_k)\cap \Dom(\ddbar^{E*}_{k,H})\cap L^2_{0,j}(X_c,L^k\otimes E)$ with respect to the graph norm of $\ddbar^E_k+\ddbar^{E*}_{k,H}$, we only need to show this inequality holds for $s\in B^{0,j}(X_c,L^k\otimes E)$ with $j\geq q$ and large $k$.	
	Suppose now $s\in B^{0,j}(X_c,L^k\otimes E)$. Let $\phi$ be in Lemma \ref{K'lem}. Thus $\phi s \in B^{0,j}(X_c,L^k\otimes E)$ with $\supp(\phi s)\subset X_v\setminus \ov{X}_u$. By Lemma \ref{keylem}, there exist $C_2>0$ and $C_3>0$ such that for $j\geq q$ and $k\geq 1$, we have
	\begin{equation}
	||\phi s||^2\leq \frac{C_2}{k}( ||\ddbar^E_k (\phi s)||^2+||\ddbar^{E*}_k (\phi s)||^2 )
	\end{equation}
	where $\chi'(\varrho)\geq C_3$ on $X_v\setminus \overline{X}_u$ and the $L^2$-norm $||\cdot||$ is given by $\omega$, $h^{L^k}_\chi$ and $h^E$ on $X_c$. Next by applying (\ref{cuteq_k}) and Lemma \ref{K'lem}, we obtain
	\begin{equation}
	||s||^2-\int_{K'}|s|^2dv_X\leq \frac{5C_2}{k}(||\ddbar^E_k s||^2+||\ddbar^{E*}_k s||^2)+\frac{12C_1C_2}{k}||s||^2.
	\end{equation}
	The proof is complete by choosing $C_0:=\max\{ 12C_1C_2,5C_2 \}>0$.
\end{proof}

 \noindent\textbf{Proof of Theorem \ref{thm_q_convex}:} Let $u_0<u<c<v$ such that $K\subset X_{u_0}\Subset X_c\Subset X_v$. From Proposition \ref{1coxFE} and Proposition \ref{prop_main}, there exists a compact subset $K'\subset X_c$ with $K\subset K'$ such that
	\begin{equation}\label{eq_h_xc}
	\begin{split}
	\limsup_{k\rightarrow \infty}n!k^{-n}\dim H_{(2)}^{j}(X_c,L^k)
	\leq \int_{K'(j,h^L_\chi)}(-1)^j c_1(L,h_\chi^L)^n\leq \int_{X_c(j,h^L_\chi)}(-1)^j c_1(L,h_\chi^L)^n.
	\end{split}
	\end{equation}
	Recall $R^L$ has at least $n-s+1$ non-negative eigenvalues on $X\setminus M$. We can suppose $K\cup M\subset X_{u_0}$ by choosing $u_0$.		We choose now $\chi=\chi(t)\in\cC^\infty(\R)$, $\chi'(t)\geq 0$, $\chi''(t)\geq 0$ for all $t\in\R$ such that  $\chi=0$ on $(-\infty,u_0)$ and $\chi'(\rho)\geq C_3>0$ on $X_v\setminus \ov X_u$. 	
	
	We have $\sqrt{-1}R^{L_\chi}=\sqrt{-1}R^L+\chi'(\varrho)\dbar\ddbar\varrho+\sqrt{-1}\chi''(\varrho)\dbar\varrho\wedge\ddbar\varrho\geq \sqrt{-1}R^L+\sqrt{-1}\chi'(\varrho)\dbar\ddbar\varrho.$	
	Note that $R^L$ has at least $n-s+1$ non-negative eigenvalues (thus at most $s-1$ negative eigenvalues) on $X\setminus (M\cup K)$; note $\chi'(\varrho)\geq 0$ on $X$ and
	$\chi'(\varrho)\geq C_3>0$ on $X_v\setminus \ov X_u$ (thus on $X_c\setminus \ov X_u$); note $\dbar\ddbar\varrho$ has at least $n-q+1$ positive eigenvalues (at most $q-1$ negative eigenvalues) on $X\setminus K$ (thus on $X_c\setminus \ov X_u$).	
	Thus for $j\geq s+q-1$ (note $s+q-1>s-1$ and $>q-1$), we have
	\begin{equation}
	X_c(j,h^L_\chi)\subset \ov X_u.
	\end{equation}
	
	Moreover, for $x\in \ov X_u\setminus X_{u_0}$, if $\chi'(\varrho)(x)>0$, then $x\notin X_c(j,h^L_\chi)$; if $\chi'(\varrho)(x)=0$, then $\sqrt{-1}R^{L_\chi}\geq \sqrt{-1}R^L$ has at least $n-s+1$ non-negative eigenvalues (thus at most $s-1$ negative eigenvalues), so  $x\notin X_c(j,h^L_\chi)$. We obtain for $j\geq s+q-1$ ($s+q-1>s-1$ and $>q-1$),
	\begin{equation}
		 X_c(j,h^L_\chi)\subset X_{u_0}.
	\end{equation}
However, by $\chi=0$ on $(-\infty,u_0)$, we have $h^L_\chi=h^L$ on $X_{u_0}$. Thus $X_c(j,h^L_\chi)=X_{u_0}(j,h_\chi^L)=X_{u_0}(j,h^L)=X_c(j,h^L)\setminus (X_c\setminus X_{u_0})(j,h^L)=X_c(j,h^L)$ for $j\geq s+q-1$. It follows that
\begin{equation}
 \int_{X_c(j,h^L_\chi)}(-1)^j c_1(L,h_\chi^L)^n
= \int_{X_c(j,h^L)}(-1)^j c_1(L,h^L)^n
\end{equation}
for $q+s-1\leq j\leq n$.	 Finally, by (\ref{eq_h_xc}), it follows that for $q+s-1\leq j\leq n$, we have
	\begin{equation}\label{eq_q_xc}
	\begin{split}
	\limsup_{k\rightarrow \infty}n!k^{-n}\dim H_{(2)}^{0,j}({X_c},{L}^k)
	\leq\int_{X_c(j,h^L)}(-1)^j c_1(L,h^L)^n= \int_{M(j)}(-1)^j c_1(L,h^L)^n.
	\end{split}
	\end{equation}
    Here the last equality is from that $X_c(j,h^L)\subset M(j,h^L)\subset X_c(j,h^L)$.
	 \cite[Theorem 3.5.6 (Hörmander),
	Theorem 3.5.7 (Andreotti-Grauert)(i), Theorem B.4.4 (Dolbeault isomorphism)]{MM07} implies, for $j\geq q$,
	$H^j(X,L^k)\cong H^j(X_v,L^k)\cong
	H^{0,j}(X_v,L^k)\cong H_{(2)}^{0,j}(X_c,L^k)$ and we apply (\ref{eq_q_xc}).
\qed

	The weak Morse inequalities for $q$-convex manifolds we treated here
	 can be viewed as asymptotic vanishing theorem for $H^q(X,L^k)$  as $k\rightarrow\infty$ when $L$ has nowhere $q$ negative and $n-q$ positive eigenvalues. It is beyond the scope of classical Bochner technique and provides explicit upper bound.

\subsection{Pseudoconvex domains}

	Let $M$ be a relatively compact domain with smooth boundary $bM$ in a complex manifold $X$. Let $\rho\in \cC^\infty(X,\R)$ such that $M=\{ x\in X: \rho(x)<0 \}$ and $d\rho\neq 0$ on $bM=\{x\in X: \rho(x)=0\}$. We denote the closure of $M$ by $\overline{M}=M\cup bM$. We say that $\rho$ is a defining function of $M$. Let $T^{(1,0)}bM:=\{ v\in T^{(1,0)}X: \dbar\rho(v)=0 \}$ be the analytic tangent bundle to $bM$. The Levi form of $\rho$ is the $2$-form $\cL_\rho:=\dbar\ddbar\rho\in \cC^\infty(bM, T^{(1,0)*}bM\otimes T^{(0,1)*}bM)$.	
	A relatively compact domain $M$ with smooth boundary $bM$ in a complex manifold $X$ is called pseudoconvex if the Levi form $\cL_\rho$ is semi-positive definite.

\noindent\textbf{Proof of Theorem \ref{thm_psc}:}
	Let $\omega$ be a Hermitian metric on $X$. Let $\rho\in \cC^\infty(X,\R)$ be a defining function of $M$ such that $M=\{x\in X:\rho(x)< 0  \}$ with $|d\rho|=1$ on the boundary $bM$. Let $x\in bM$. Let $\{ w_j \}_{j=1}^n$ be a local orthonormal frame of $T^{(1,0)}X$ with dual frame $\{ w^j\}_{j=1}^n$ of $T^{(1,0)*}X$ around $x$, such that  $T^{(1,0)}_xbM$ is generated by $\{w_2,\cdots,w_n\}$ and $(\dbar\ddbar\rho)(w_k,\ov w_j)_x=\delta_{kj}a_j(x)$ for $j, k\geq 2$. Thus $a_j(x)\geq 0$ for each $2\leq j\leq n$ as $M$ is pseudoconvex. For $s\in B^{0,q}(M,L^k\otimes E)=\{ s\in \Omega^{0,q}(\ov M,L^k\otimes E): i_{e_{\bn}^{(0,1)}}s =0~\mbox{on}~bM \}$, we have $i_{\ov w_1}s(x)=0$.
	So $ \cL_\rho(s,s)(x)
	=\sum_{j,k=2}^n(\dbar\ddbar\rho)(w_k,\ov{w}_j)\langle \ov{w}^j\wedge i_{\ov{w}_k}s(x),s(x)\rangle_h
	=\sum_{j=2}^n a_j(x) \langle \ov w^j\wedge i_{\ov w_j} s(x),s(x)\rangle_h
	\geq 0$, it follows that
	\begin{equation}
	\int_{bM} \cL_\rho(s,s)dv_{bM}\geq 0.
	\end{equation}
	Let $X_c:=\{ x\in X:\rho(x)<c \}$ for $c\in \R$.
	We fix $u<0<v$ such that $L>0$ on a open neighbourhood of $X_v\setminus \ov{X}_u$, then there exists $C_L>0$ such that for any $q\geq 1$ and any holomorphic Hermitian vector bundle $(F,h^F)$ on $X$,
	\begin{equation}
	\langle R^L(w_l,\ol{w}_k)\ol{w}^k\wedge i_{\ol{w}_l}s,s \rangle_h\geq C_L|s|^2, \quad s\in \Omega^{0,q}_0(X_v\setminus \ov{X}_u,F).
	\end{equation}
	As in (\ref{eq_BKN}), there exist $C_4\geq 0$ and $C_5\geq 0$ such that for any $s\in B^{0,q}(M,L^k\otimes E)$ with $\supp (s)\in X_v\setminus\ov{X}_u$ and $q\geq 1$,
	\begin{equation}
	\begin{split}
	\frac{3}{2}(||\ddbar^E_k s||^2+||\ddbar^{E*}_k s||^2)
	&\geq \langle  R^{L^k\otimes E\otimes K^*_X}(w_j,\ov{w}_k)\ov{w}^k\wedge i_{\ov{w}_j} s,s\rangle+\int_{bM}\cL_{\rho}(s,s)dv_{bM}-C_4||s||^2\\
	&\geq  \int_M(kC_L-C_5-C_4)|s|^2dv_X.
	\end{split}
	\end{equation}
	For any $k\geq k_0:=[2\frac{C_4+C_5}{C_L}]+1$, we have $C_L-\frac{C_4+C_5}{k}\geq \frac{1}{2}C_L$. Let $C_2:=\frac{3}{C_L}$. As in Lemma \ref{keylem}, for any $s\in B^{0,q}(M,L^k\otimes E)$ with $\supp(s)\subset X_v\setminus \overline{X}_u$, $q\geq 1$ and $k\geq k_0>0$, we have
	\begin{equation}
	||s||^2\leq \frac{C_2}{k}( ||\ddbar^E_k s||^2+||\ddbar^{E*}_k s||^2 ),
	\end{equation}
	where the $L^2$-norm $||\cdot||$ is given by $\omega$, $h^{L^k}$ and $h^E$ on $M$.
	Along the same argument in Lemma \ref{1coxFE}, we conclude that
	there exist a compact subset $K'\subset M$ (Indeed, $K':=\ov{\{ \rho<-\epsilon \}}$ for some $\epsilon>0$ such that $\{ -\epsilon<\rho<\epsilon \}\Subset X_v\setminus \ov{X}_u$ as in Lemma \ref{K'lem}) and $C_0>0$ such that for sufficiently large $k$, we have
	\begin{equation}
	(1-\frac{C_0}{k})||s||^2\leq \frac{C_0}{k}(||\ddbar^E_ks||^2+||\ddbar^{E*}_{k}s||^2)+\int_{K'} |s|^2 dv_X
	\end{equation}
	for any $s\in \Dom(\ddbar^E_k)\cap \Dom(\ddbar^{E*}_{k})\cap L^2_{0,q}(M,L^k\otimes E)$ and each $1\leq q \leq n$,
	where the $L^2$-norm is given by $\omega$, $h^{L^k}$ and $h^E$ on $M$.
	Finally, we can apply Proposition \ref{prop_main} on $M$ and note $K'\subset M$.
\qed

\subsection{Weakly $1$-complete manifolds}

	A complex manifold $X$ is called weakly $1$-complete if there exists a smooth plurisubharmonic function $\varphi\in \cC^\infty(X,\R)$ such that $\{x\in X: \varphi(x)<c\}\Subset X$ for any $c\in \R$. $\varphi$ is called an exhaustion function.

\begin{corollary}
	Let $X$ be a weakly $1$-complete manifold of dimension $n$. Let $(L,h^L)$ and $(E,h^E)$ be holomorphic Hermitian line bundles on $X$. Let $(L,h^L)$ be positive on $X\setminus K$ for a compact subset $K\subset X$.
	Then for any $q\geq 1$, we have
	\begin{equation}
	\limsup_{k\rightarrow \infty}n!k^{-n}\dim H^q(X,L^k\otimes E)\leq \int_{K(q)}(-1)^q c_1(L,h^L)^n.
	\end{equation}
\end{corollary}

\begin{proof}
	Let $\varphi\in \cC^\infty(X,\R)$ be an exhaustion function of $X$ such that $\sqrt{-1}\dbar\ddbar\varphi\geq 0$ on $X$ and $X_c:=\{\varphi<c\}\Subset X$ for all $c\in \R$. We choose  a regular value $c\in \R$ of $\varphi$ such that $K\subset X_c$ by  Sard's theorem. Thus $X_c$ is a smooth pseudoconvex domain and $L>0$ on a neighbourhood of $bX_c$. We apply Theorem \ref{thm_psc} on $X_c$.
	Finally, we use
	$
	H^{q}(X,L^k\otimes E)\cong H^q(X_c,L^k\otimes E)\cong H^{0,q}(X_c,L^k\otimes E)\cong\cH^{0,q}(X_c,L^k\otimes E)
	$ for $q\geq 1$ and sufficiently large $k$ by \cite[Theorem 6.2]{Takegoshi:83} (see also \cite[Theorem 3.5.11]{MM07}).
\end{proof}

More generally, we reprove the main result in \cite{M:92} without complete metric. Recall $(L,h^L)$ is called Griffiths $q$-positive at $x\in X$, if the curvature form $R_x^{L}$ has at least $n-q+1$
positive eigenvalues.

\noindent\textbf{Proof of Theorem \ref{thm_w1c}:}
	Firstly,
	let $\varphi\in \cC^\infty(X,\R)$ be an exhaustion function of $X$ such that $\sqrt{-1}\dbar\ddbar\varphi\geq 0$ on $X$ and $X_r:=\{\varphi<r\}\Subset X$ for all $r\in \R$. We assume $c$ is a regular value of $\varphi$ such that $K\subset X_c$. Thus $X_c$ is a smooth pseudoconvex domain. Let $\rho$ be a defining function of $X_c=M:=\{\rho(x)<0\}$ with $|d\rho|=1$ on the boundary $bM$. For $s\in B^{0,j}(bX_c,L^k\otimes E)$ with $j\geq 1$,
	\begin{equation}
		\int_{bX_c} \cL_\rho(s,s)dv_{bX_c}\geq 0.
	\end{equation}
	Secondly,
	let $u<c<v$ such that $K\subset X_{u}$ such that $L$ is Griffith $q$-positive on a neighbourhood $X_v\setminus \ov X_{u}$ of $bX_c$. As Lemma \ref{lowbd_rho_lem}, there exists a metric $\omega$ on $X$ and $C_1>0$ such that
	\begin{equation}
		\langle R^L(w_l,\ol{w}_k)\ol{w}^k\wedge i_{\ol{w}_l}s,s \rangle_h\geq C_1|s|^2, \quad s\in \Omega^{0,j}_0(X_v\setminus \overline{X}_u,L^k\otimes E), j\geq q.
	\end{equation}
	Thirdly, as \eqref{eq_posline} there exits $C_2>0$, such that for $s\in B^{0,j}(X_c,L^k\otimes E)$ with $\supp(s)\subset X_v\setminus \overline{X}_u$, $j\geq q$ and $k\geq k_0>0$, with respect to $\omega$, $h^{L}$ and $h^E$ on $X_c$,
	\begin{equation}
		||s||^2\leq \frac{C_2}{k}( ||\ddbar^E_k s||^2+||\ddbar^{E*}_k s||^2 ).
	\end{equation}
	The rest proof is as same as Theorem \eqref{thm_psc}, that is,
	$	(1-\frac{C_0}{k})||s||^2\leq \frac{C_0}{k}(||\ddbar^E_ks||^2+||\ddbar^{E*}_{k}s||^2)+\int_{K'} |s|^2 dv_X$
	for any $s\in \Dom(\ddbar^E_k)\cap \Dom(\ddbar^{E*}_{k})\cap L^2_{0,q}(X_c,L^k\otimes E)$ and each $1\leq q \leq n$,
	where the $L^2$-norm is given by $\omega$, $h^{L^k}$ and $h^E$ on $X_c$. Note that for $j\geq q$, $K(j)=K'(j)$ by the Griffith $q$-positive.
	
	In particular, if $L>0$ on $X\setminus K$, then for $j\geq 1$, we  use
	$
	H^{j}(X,L^k\otimes E)\cong\cH^{0,j}(X_c,L^k\otimes E)=H_{(2)}^{j}(X_c,L^k\otimes E)
	$ for sufficiently large $k$ by \cite[Theorem 6.2]{Takegoshi:83} (see \cite[Theorem 3.5.11]{MM07}).
\qed

\subsection{Complete manifolds}

A Hermitian manifold $(X,\omega)$ is called complete, if all geodesics are defined for all time on the underlying Riemannian manifold. If $(X,\omega)$ is complete, for arbitrary holomorphic Hermitian vector bundle $(F,h^F)$ on $X$, $\Omega_0^{0,\bullet}(X,F)$ is dense in $\Dom(\ddbar^F)$, $\Dom(\ddbar^{F*}_H)$ and $\Dom(\ddbar^F)\cap \Dom(\ddbar^{F*}_H)$ in the graph-norms of $\ddbar^F$, $\ddbar^{F*}_H$ and $\ddbar^E+\ddbar^{E*}_H$ respectively, see \cite[Lemma 3.3.1 (Andreotti-Vesentini), Corollary 3.3.3]{MM07}. Here the graph-norm is  $\|s\|+\|Rs\|$ for $s\in \Dom(R)$.

\begin{lemma}\label{lem_complete}
	Let $(X,\omega)$ be a complete Hermitian manifold of dimension $n$. Let $(L,h^L)$ be a holomorphic Hermitian line bundle  on $X$ such that $\omega=c_1(L,h^L)$ on $X\setminus M$ for a compact subset $M$. Then there exist $C_0>0$ and $M\Subset M'$ such that for each $1\leq q\leq n$, we have for sufficiently large $k$,
	\begin{equation}
	\left( 1-\frac{C_0}{k} \right)\|s\|^2\leq \frac{C_0}{k}\left( \|\ddbar^{K_X}_k s\|^2+\|\ddbar_k^{K_X*} s\|^2 \right)+\int_{M'}|s|^2dv_X
	\end{equation}
	for $s\in\Dom(\ddbar^{K_X}_{k})\cap\Dom(\ddbar^{K_X*}_{k})\cap L^2_{n,q}(X,L^k)$.
\end{lemma}

\begin{proof}
	Since $(X,\omega)$ is complete, for arbitrary holomorphic Hermitian vector bundle $(E,h^E)$ the Hilbert space adjoint and the maximal extension of the formal adjoint of $\ddbar^E_k$ coincide, $\ddbar^{E*}_{k,H}=\ddbar^{E*}_k$.  Let $ \Lambda$ be the adjoint of the operator $\omega\wedge\cdot$ with respect to the Hermitian inner product induced by $\omega$ and $h^L$. In a local orthonormal frame $\{ \omega_j \}_{j=1}^n$ of $T^{(1,0)}X$ with dual frame $\{ w^j\}_{j=1}^n$ of $T^{(1,0)*}X$, $\omega=\sqrt{-1}\sum_{j=1}^n \omega^j\wedge\ov\omega^j$ and $\Lambda=-\sqrt{-1}i_{\ov w_j}i_{w_j}$. Thus $\sqrt{-1}R^{(L,h^L)}=\sqrt{-1}\sum_{j=1}^n \omega^j\wedge\ov\omega^j$ outside $M$.
	Let $\{e_k\}$ be a local frame of $L^k$. For $s\in \Omega^{n,q}_0(X\setminus M,L^k)$, we can write $s=\sum_{|J|=q} s_J\omega^1\wedge\cdots\wedge\omega^n\wedge\ov\omega^J\otimes e_k$ locally, thus
	\begin{equation}
	[\sqrt{-1}R^L,\Lambda]s
	=\sum_{|J|=q}(q s_J\omega^1\wedge\cdots\wedge\omega^n\wedge\ov\omega^J)\otimes e_k
	= qs.
	\end{equation}
	Since $(X\setminus M, \sqrt{-1}R^{(L,h^L)})$ is K\"{a}hler, we apply \eqref{eq_bkn} for $s\in \Omega^{n,q}_0(X\setminus M,L^k)$ with $1\leq q\leq n$, and obtain
	$\|\ddbar_k s\|^2+\|\ddbar^{*}_k s\|^2=
	\langle \square^{L^k}s,s \rangle
	\geq k \langle  [\sqrt{-1} R^L,\Lambda]s,s \rangle\geq  qk\|s\|^2\geq k\|s\|^2$.
	So we have
	\begin{equation}
	\|s\|^2\leq \frac{1}{k}( \|\ddbar_k s\|^2+\|\ddbar^{*}_k s\|^2 ).
	\end{equation}
	
	Next we follow the analogue argument in Proposition \ref{1coxFE}. Let $V$ and $U$ be open subsets of $X$ such that $M\subset V\Subset U\Subset X$. We choose a function $\xi\in \cC^\infty_0(U,\R)$ such that $0\leq \xi\leq 1$ and $\xi\equiv 1$ on $\ov V$. We set $\phi:=1-\xi$, thus $\phi\in \cC^\infty(X,\R)$ satisfying $0\leq \phi \leq 1$ and $\phi \equiv 0$ on $\ov V$.
	Now let $s\in \Omega_0^{n,q}(X,L^k)$, thus $\phi s\in \Omega^{n,q}_0(X\setminus M, L^k)$. We set $M':=\ov U$, then
	\begin{equation}
	\|\phi s\|^2\geq \|s\|^2-\int_{K'}|s|^2dv_X,
	\end{equation}
	and there exists a constant $C_1>0$ such that	
	\begin{equation}
	\frac{1}{k}(\|\ddbar_k (\phi s)\|^2+\|\ddbar^{*}_k(\phi s)\|^2)\leq \frac{5}{k}(\|\ddbar_k s\|^2+\|\ddbar^{*}_k s\|^2)+\frac{12C_1}{k}\|s\|^2.
	\end{equation}		
	By combining the above three inequalities, there exists $C_0>0$ such that for any $s\in \Omega^{n,q}_0(X,L^k)=\Omega^{0,q}_0(X,L^k\otimes K_X)$ with $1\leq q\leq n$ with $k$ large enough,
	\begin{equation}
	(1-\frac{12C_1}{k})\|s\|^2\leq \frac{5}{k}(\|\ddbar_k s\|^2+\|\ddbar^{*}_k s\|^2)+\int_{M'}|s|^2dv_X.
	\end{equation}
	Finally choose $C_0=\max\{ 12C_1, 5 \}$. The assertion follows from the fact that $\Omega_0^{0,\bullet}(X,L^k\otimes K_X)$ is dense in $\Dom(\ddbar_k^{K_X})\cap \Dom(\ddbar_k^{K_X*})$ in the graph-norm.
\end{proof}

\noindent\textbf{Theorem \ref{thm_complete}:} Let $q\geq 1$.
	From Lemma \ref{lem_complete},
	the optimal fundamental estimate holds in bidegree $(0,q)$ for forms with values in $L^k\otimes K_X$ for $k$ large,  then use Proposition \ref{prop_main} and note that $M(q)=M'(q)$.
\qed

\subsection{Covering manifolds}
Let $(\widetilde{X}, J)$ be a complex manifold of dimension $n$ with a compatible Riemannian metric $g^{T\widetilde{X}}$. Let $\widetilde \omega$ be the associated real $(1,1)$-form defined by $\widetilde \omega(X,Y)=g^{T\widetilde{X}}(J X,Y)$ on ${T\widetilde{X}}$.
A group $\Gamma$ is called a discrete group acting holomorphically, freely and properly on $\widetilde{X}$,
if $\Gamma$ is equipped with the discrete topology such that (1) the map  $\Gamma \times \widetilde{X} \rightarrow \widetilde{X}, (r, \widetilde{x})\mapsto r.\widetilde{x}$ is holomorphic; (2) $r.\widetilde{x}=\widetilde{x}$ for some $\widetilde{x} \in \widetilde{X}$ implies that $r=e$ the unit element of $\Gamma$; and (3) the map  $\Gamma \times \widetilde{X} \rightarrow \widetilde{X}\times \widetilde{X}, (r,\widetilde{x})\rightarrow (r.\widetilde{x},\widetilde{x})$ is proper.
A Riemannian metric $g^{T\widetilde{X}}$ (or $\widetilde{\omega}$) is called $\Gamma$-equivariant, if the map $r:\widetilde{X}\rightarrow \widetilde{X}$ is isometric with respect to $g^{T\widetilde{X}}$
for every $r\in \Gamma$.

We say a Hermitian manifold $(\widetilde{X},\widetilde{\omega})$ is a covering manifold, if there exists a discrete group $\Gamma$ acting holomorpically, freely and properly on $\widetilde{X}$ such that $\widetilde{\omega}$ is $\Gamma$-equivariant and the quotient $X:=\widetilde{X}/\Gamma$ is compact. An relatively compact open subset $U \Subset \widetilde{X} $ is called a fundamental domain of the action $\Gamma$ on a covering manifold $ \widetilde{X}$, if the following conditions are satisfied: (a) $ \widetilde{X}=\cup_{r\in\Gamma}r(\overline{U}) $; (b) $r_1(U)\cap r_2(U)$ is empty for $ r_1,r_2 \in \Gamma$ with $r_1\neq r_2$; and (c) $\overline{U}\setminus U$ has zero measure.
The fundamental domain exists, and it can be constructed (e.g. \cite[3.6]{MM07}).
A holomorphic Hermitian vector bundle $(\widetilde{F},h^{\widetilde{F}})$ over $\widetilde{X}$ is called $\Gamma$-invariant, if there is a map $r_{\widetilde{F}}:{\widetilde{F}}\rightarrow {\widetilde{F}}$ associated to $r\in \Gamma$, which commutes with the fibre projection $\pi:{\widetilde{F}}\rightarrow \widetilde{X}$ (i.e., $r\circ\pi =\pi\circ r_{\widetilde{F}}$), such that $h^{\widetilde{F}}(v,w)=h^{\widetilde{F}}(r_{{\widetilde{F}}}v,r_{{\widetilde{F}}}w)$ for $v,w \in \widetilde{F}$.

\noindent\textbf{Proof of Theorem \ref{thm_cover}:}
Let $U\Subset \widetilde X$ be the fundamental domain of the action $\Gamma$. Due to the construction of $U$, we have a biholomorphic map $\pi_\Gamma|_U: U\rightarrow X\setminus Z$,
$\pi_{\Gamma}(U)=X\setminus Z$
with a zero measure subset $Z\subset X$. In fact, we choose a finite open cover $\{ U_j \}_{j=1}^N$ of $X$ such that, there exist open subsets $\{ \widetilde{U}_j \}_{j=1}^N$ in $\widetilde{X}$ satisfying $\pi_{\Gamma}: \widetilde{U}_j\rightarrow U_j$ is biholomorphic. Define $W_1=U_1, W_2=U_2\setminus \overline{U}_1,\cdots,W_j=U_j\setminus (\overline{U}_1\cup\cdots\cup \overline{U}_{j-1}),\cdots,W_N=U_N\setminus (\overline{U}_1\cup\cdots\cup \overline{U}_{N-1})$. The fundamental domain $U$ is given by $U=\cup_{j=1}^N \pi_{\Gamma}^{-1}(W_j)$. Thus $\pi_{\Gamma}(U)=\cup_{j=1}^N W_j=X\setminus Z$, where $Z\subset \cup_{j=1}^N \partial U_j$ is of measure zero.

Let $\{\widetilde{s}^k_j\}_{j}$ be an orthonormal basis of $\cH^{0,q}(\widetilde{X},\widetilde{L}^k)$ and let $\widetilde{B}^{q}_k$ be
the Bergman kernel function defined by
$\widetilde{B}^{q}_k(\widetilde{x})=\sum_{j}|s^k_j(\widetilde{x})|,
\;\widetilde{x}\in \widetilde{X}$,
where $|\cdot|:=|\cdot|_{\widetilde{h}_k,\widetilde{\omega}}$ is the point-wise Hermitian norm. By \cite[Lemma 3.6.2]{MM07}, we have for each $0\leq q\leq n$,
\begin{equation}
\dim_{\Gamma}\cH^{0,q}(\widetilde{X},\widetilde{L}^k)=\sum_{j}\int_U |s^k_j(\widetilde{x})|^2 dv_{\widetilde{X}}(\widetilde{x})= \int_U \sum_{j} |s^k_j(\widetilde{x})|^2 dv_{\widetilde{X}}(\widetilde{x})=\int_U \widetilde{B}^{q}_k(\widetilde{x}) dv_{\widetilde{X}}(\widetilde{x}),
\end{equation}
where $dv_{\widetilde{X}}=\widetilde{\omega}^n/n!$. By applying (\ref{eq_berman_mz}) on $\widetilde{X}$, we have for each $0\leq q\leq n$,
\begin{equation}
\limsup_{k\rightarrow \infty}k^{-n}\widetilde{B}^{q,k}(\widetilde{x})\leq (-1)^q 1_{\widetilde{X}(q)}\frac{c_1(\widetilde{L},h^{\widetilde{L}})^n}{\widetilde{\omega}^n}(\widetilde{x}),\quad \widetilde{x}\in\widetilde{X}.
\end{equation}
 Analogue to (\ref{eq_fatou}), by Fatou's lemma and the fact that $\overline{U}\subset \widetilde{X}$ is compact with finite volume,
 \begin{equation}
 \limsup_{k\rightarrow \infty} \int_U k^{-n} \widetilde{B}^{q}_k(\widetilde{x}) dv_{\widetilde{X}}(\widetilde{x})
 \leq \int_U \limsup_{k\rightarrow \infty}k^{-n}\widetilde{B}^{q}_k(\widetilde{x})dv_{\widetilde{X}}(\widetilde{x}).
 \end{equation}
   Finally the weak Morse inequalities on covering manifolds follows by
\begin{equation}\nonumber
\begin{split}
&\limsup_{k\rightarrow \infty}k^{-n}\dim_{\Gamma}\cH^{0,q}(\widetilde{X},\widetilde{L}^k)
= \limsup_{k\rightarrow \infty}k^{-n} \int_U \widetilde{B}^{q}_k(\widetilde{x}) dv_{\widetilde{X}}(\widetilde{x})
\leq \int_U \limsup_{k\rightarrow \infty}k^{-n}\widetilde{B}^{q}_k(\widetilde{x})dv_{\widetilde{X}}(\widetilde{x})\\
&\leq \frac{1}{n!}\int_U(-1)^q 1_{\widetilde{X}(q)}c_1(\widetilde{L},h^{\widetilde{L}})^n=\frac{1}{n!}\int_{U(q)}(-1)^q c_1(\widetilde{L},h^{\widetilde{L}})^n
=\frac{1}{n!}\int_{(\pi_{\Gamma}^{-1}(X\setminus Z))(q)}(-1)^q \pi_\Gamma^*c_1(L,h^L)^n\\&
=\frac{1}{n!}\int_{(X\setminus Z)(q)}(-1)^q c_1(L,h^L)^n=\frac{1}{n!}\int_{X(q)}(-1)^q c_1(L,h^L)^n.
\end{split}
\end{equation}
\qed

\section{Asymptotics of spectral function of lower energy forms}\label{Sec_specfct}

In the previous section we only considered the kernel space of Kodaira Laplacian, the strength of the optimal fundamental estimate has not been fully exploited. Based on the in-depth analysis of the spectral function of Kodaira Laplacian \cite{HM:14}, we present further remarks in this section.
Let $(X,\omega)$ be a Hermitian manifold of dimension $n$ and let $(L,h^L)$ and $(E,h^E)$ be holomorphic Hermitian line bundles on $X$. Let $\square^E_k$ be the Gaffney extension of Kodaira Laplacian. Let $0\leq q\leq n$.
Let $E^q_{\leq k^{-N_0}}(\square^E_k): L^2_{0,q}(X,L^k\otimes E)\rightarrow \cE^q(k^{-N_0},\square^E_k):=\Im E^q_{\leq k^{-N_0}}(\square^E_k)$ be the spectral projection of $\square^E_k$. Let $\{s_j\}_{j\geq 1}$ be an orthonormal frame of $\cE^q(k^{-N_0},\square_k^E)$ and
let $B^q_{\leq k^{-N_0}}(x):=\sum_{j}|s_j(x)|^2_h$ be the spectral function. We denote the spectrum counting function of $\square^E_k$ by
$N^q(k^{-N_0},\square_k^E):=\dim \cE^q(k^{-N_0},\square_k^E).$
By the spectral theorem,
\begin{equation}
\cH^{0,q}(X,L^k\otimes E)=\cE^q(0,\square_k^E)\subset\cE^q(k^{-N_0},\square_k^E)\subset \Dom(\square_k^E)\cap L^2_{0,q}(X,L^k\otimes E).
\end{equation}

\begin{theorem}[{\cite[Corollary 1.4]{HM:14}}]
	Let $N_0\geq 2n+1$ and $0\leq q\leq n$. The spectral function
	of the Kodaira Laplacian has the following asymptotic bahaviour:
	\begin{equation}
	\begin{split}
	\limsup_{k\rightarrow \infty}k^{-n}B^q_{\leq k^{-N_0}}(x)= (-1)^q 1_{{X}(q)}\frac{c_1({L},h^{{L}})^n}{{\omega}^n}({x}).
	\end{split}
	\end{equation}
	For any compact subset $K\subset X$, there exist $C>0$ and $k_0>0$ such that for any $k\geq k_0$ and any $x\in K$,
	\begin{equation}
	k^{-n}B^q_{\leq k^{-N_0}}(x)\leq C.
	\end{equation}
\end{theorem}

As a consequence, we have the observation analogue to Proposition \ref{prop_main} as follows.
\begin{proposition}\label{prop_main_hm}
	Let $0\leq q\leq n$.
	Suppose there exist a compact subset $K\subset X$ and $C_0>0$ such that,
	for sufficiently large $k$, we have
	\begin{equation}
	\left(1-\frac{C_0}{k}\right)||s||^2\leq \frac{C_0}{k}\left(||\ddbar^E_ks||^2+||\ddbar^{E*}_{k}s||^2\right)+\int_{K} |s|^2 dv_X
	\end{equation}
	for $s\in \Dom(\ddbar^E_k)\cap \Dom(\ddbar^{E*}_{k})\cap L^2_{0,q}(M,L^k\otimes E)$.
	Then  we have
	\begin{equation}
	\limsup_{k\rightarrow \infty}n!k^{-n}N^q(k^{-N_0},\square_k^E)\leq \int_{K(q)}(-1)^q c_1(L,h^L)^n,
	\end{equation}
\end{proposition}
\begin{proof}
	By applying Fatou's lemma to the sequence of non-negative functions $C-k^{-n}B^q_{\leq k^{-N_0}}$ and the finiteness of the volume $\int_K C dv_X<\infty$, we have
	\begin{equation}
	\limsup_{k\rightarrow \infty} \int_K k^{-n} {B}^{q}_{\leq k^{-N_0}}({x}) dv_{{X}}({x})
	\leq \int_K \limsup_{k\rightarrow \infty}k^{-n}{B}^{q}_{\leq k^{-N_0}}({x})dv_{{X}}({x}).
	\end{equation}
	From the assumptions for sufficiently large $k$, for any $s\in \cE^q(k^{-N_0},\square_k^E)\subset \Dom(\square_k^E)\cap L^2_{0,q}(X,L^k\otimes E)$,
	\begin{equation}
	(1-\frac{C_0}{k})||s||^2\leq \frac{C_0}{k}(||\ddbar^E_ks||^2+||\ddbar^{E*}_{k,H}s||^2)+\int_{K} |s|^2 dv_X\leq C_0k^{-N_0-1}\|s\|^2+\int_{K} |s|^2 dv_X.
	\end{equation}
	Set $a_k:=\frac{1}{1-C_0k^{-1}-C_0k^{-N_0-1}}$. Thus $\|s\|^2\leq a_k\int_{K} |s|^2 dv_X$. Finally, we obtain
	\begin{equation}
	\begin{split}
	&\limsup_{k\rightarrow \infty}\left(k^{-n}\dim \cE^q(k^{-N_0,\square_k^E})\right)\leq\limsup_{k\rightarrow \infty}\left(k^{-n}a_k \int_{K}B_{\leq k^{-N_0}}^q(x)dv_X(x)\right)\\
	&\leq\left(\limsup_{k\rightarrow \infty} a_k\right)\left(\limsup_{k\rightarrow \infty} \int_{K}k^{-n}B_{\leq k^{-N_0}}^q(x)dv_X(x)\right)\leq\int_{K}\limsup_{k\rightarrow \infty}k^{-n} B_{\leq k^{-N_0}}^q(x)dv_X(x)\\
	&\leq \int_{K(q)}(-1)^q \frac{c_1({L},h^{{L}})^n}{n!}
	\end{split}
	\end{equation}
\end{proof}

Analogue to Theorem \ref{thm_complete} we have:
\begin{corollary}
	Let $(X,\Theta)$ be a complete Hermitian manifold of dimension $n$. Let $(L,h^L)$ be a holomorphic Hermitian line bundle  on $X$ such that $\Theta=c_1(L,h^L)$ on $X\setminus M$ for a compact subset $M$. Then for each $1\leq q\leq n$, $N_0\geq 2n+1$, we have
	\begin{equation}
	\limsup_{k\rightarrow \infty}n!k^{-n}N^q(k^{-N_0},\square_k^{K_X})
	\leq
	\int_{M(q)}(-1)^q c_1(L,h^L)^n.
	\end{equation}
\end{corollary}

Let $(\widetilde X,\widetilde \omega)$ be a Hermitian manifold of dimension $n$ on which a discrete
group $\Gamma$ acts holomorphically, freely and properly such that $\widetilde{\omega}$ is a $\Gamma$-invariant Hermitian
metric and the quotient $X:=\widetilde X/\Gamma$ is compact. Let $(\widetilde L,h^{\widetilde L})$ be a $\Gamma$-invariant
holomorphic Hermitian line bundles on $X$.  Let $\pi_{\Gamma}:\widetilde X\rightarrow X=\widetilde X/\Gamma$ be the projection. Let $\widetilde{B}_{\leq k^{-N_0}}^{q}$ be the spectral function of the Laplacian $\widetilde{\square}^E_k$. We see
\begin{equation}
N_\Gamma(k^{-N_0},\widetilde{\square}^E_k):=\dim_{\Gamma}\cE^q(k^{-N_0},\widetilde{\square}^E_k)=\int_U \widetilde{B}_{\leq k^{-N_0}}^{q}(\widetilde{x}) dv_{\widetilde{X}}(\widetilde{x}).
\end{equation}	
Analogue to Theorem \ref{thm_cover} and Proposition \ref{prop_main_hm} we have:
\begin{corollary}
	For $0\leq q\leq n$, $N_0\geq 2n+1$, we have
	\begin{equation}
	\limsup_{k\rightarrow \infty}n!k^{-n}N_\Gamma(k^{-N_0},\widetilde{\square}^E_k)\leq \int_{X(q)}(-1)^q c_1(L,h^L)^n.
	\end{equation}
\end{corollary}

\bibliographystyle{amsalpha}

\end{document}